\documentclass{article}
\usepackage[utf8]{inputenc}

\usepackage{amsfonts,amsmath,amssymb,amsthm,authblk}
\usepackage[square,sort&compress,comma,numbers]{natbib}
\usepackage{units,xcolor}

\newtheorem{theorem}{Theorem}
\newtheorem{prop}[theorem]{Proposi{t}ion}
\newtheorem{lemma}[theorem]{Lemma}

\theoremstyle{plain}
\newtheorem{corollary}[theorem]{Corollary}

\theoremstyle{remark}
\newtheorem{defn}[theorem]{Def{i}ni{t}ion}

\newcommand{\abs}[1]{\ensuremath{\lvert #1 \rvert}}
\newcommand{\eps}{\ensuremath{\varepsilon}}
\newcommand{\F}{\ensuremath{\mathbb{F}}}
\newcommand{\K}{\ensuremath{K}}
\newcommand{\N}{\ensuremath{\mathbb{N}}}
\renewcommand{\O}{\ensuremath{\mathcal{O}}}
\newcommand{\ord}{\ensuremath{\operatorname{ord}}}
\newcommand{\Q}{\ensuremath{\mathbb{Q}}}
\newcommand{\R}{\ensuremath{\mathbb{R}}}
\newcommand{\Z}{\ensuremath{\mathbb{Z}}}

\makeatletter

\newcommand{\Rmnum}[1]{\expandafter\@slowromancap\romannumeral #1@}
\makeatother

\makeatletter
\newcommand{\subjclass}[2][1991]{%
  \let\@oldtitle\@title%
  \gdef\@title{\@oldtitle\footnotetext{#1 \emph{Mathemat{i}cs Subject Classif{i}cat{i}on.} #2.}}%
}
\newcommand{\keywords}[1]{%
  \let\@@oldtitle\@title%
  \gdef\@title{\@@oldtitle\footnotetext{\emph{Key words and phrases.} #1.}}%
}
\makeatother

\title{\protect{On the coef{f}icient-choosing game}}

\author[1]{Divyum Sharma}
\affil[1]{Department of Mathematics\\
Birla~Inst{i}tute~of~Technology~and~Science, Pilani 333\,031 \textsc{India}
\texttt{divyum.sharma\symbol{64}pilani.bits-pilani.ac.in}}

\author[2]{L.~Singhal}
\affil[2]{Yau Mathemat{i}cal Sciences Center\\
Tsinghua~University,~Beijing~100\,084 \textsc{China}
\texttt{singhal\symbol{64}tsinghua.edu.cn}}

\date{}


\begin{document}

\maketitle

\begin{abstract}
  Nora and Wanda are two players who choose coef{f}icients of a degree $d$ polynomial from some f{i}xed unital commutat{i}ve ring $R$. Wanda is declared the winner if the polynomial has a root in the ring of fract{i}ons of $R$ and Nora is declared the winner otherwise. We extend the theory of these games given by \citeauthor*{GWZ18} to all f{i}nite cyclic rings and determine the possible outcomes. A family of examples is also constructed using discrete valuat{i}on rings for a variant of the game proposed by these authors. Our techniques there lead us to an adversarial approach to construct{i}ng rat{i}onal polynomials of any prescribed degree (equal to $3$ or greater than $8$) with no roots in the maximal abelian extension of \Q.
\end{abstract}

\section{Introduction}\label{S:intro}
  Let $R$ be a commutat{i}ve ring with unity. \citeauthor*{GWZ18}~\cite{GWZ18} recently introduced a two-player game in which the players, Nora and Wanda, take turns to pick coef{f}icients of a degree $d$ polynomial from $R$. The leading coef{f}icient $a_d$ (say) and the constant coef{f}icient $a_0$ are not allowed to be zero. Nora is said to win if the result{i}ng polynomial has no roots in the ring of fract{i}ons $\operatorname{Frac} (R)$ and Wanda wins otherwise. The authors exhibited many instances of the game over integral domains along with possible winning strategies.  They proved that if $R$ is a subring of a number field, then the last player can always win. The proof used results about the number of solutions of $S$-unit equations. Over the f{i}eld of real numbers, they gave a winning strategy for Player\,\Rmnum{1} in the case of quadratic polynomials while Wanda can win in all remaining cases. When $R$ is a f{i}nite field, they established that Wanda has a winning strategy for degree $3$ and $\textrm{char}(R)= 3$, and that the last player can always win in the remaining cases. \citeauthor{Dic58}'s classi{f}icat{i}on of permutation polynomials of degree~$3$ was used towards this end~\cite{Dic58,GWZ18}.\\[-0.1cm]
  
 One aim of this paper is to extend the theory of these games beyond integral domains. More precisely, we consider what happens when $R$ is a f{i}nite cyclic ring of order $N$. It is noted that the ring of fract{i}ons of any such ring is isomorphic to the ring itself. Furthermore, we shall conf{i}ne our attent{i}on to $N$ not being a prime since $R$ is a f{i}nite f{i}eld otherwise and the game over such rings has already been sat{i}sfactorily resolved in~\cite{GWZ18}. Here, we prove
  \begin{theorem}\label{Th:IntroCyclic}
    Let $d \in \N^*,\ N > 1$ be a composite number and $R = \Z / N\Z$. Then, the last player can win when
    \begin{enumerate}
      \item $d$ is even, or
      \item $d > 1$ and $N$ is cube-free, or
      \item $d > 3$ with $N = 16 N_2$ for some cube-free odd integer $N_2$.
    \end{enumerate}
    In all other cases, Wanda always has a winning strategy over $R$.
  \end{theorem}
   We will soon explain in Lemmata~\ref{L:lin} and \ref{L:GWZ} as to how can Wanda win when $d = 1$ and/or when she is the last player. For $d > 1$, Nora's strategy for cube-free composite numbers is given in Lemma~\ref{L:cfree} which gets strengthened later in Corollary~\ref{C:last} for $d > 3$. Wanda's route to an ensured victory for the remaining parameter values of $d$ and $N$, in spite of not being the last player, follows from Corollaries~\ref{C:CRT} and \ref{C:Wanda}.\\[-0.1cm]

  There is a related variant of the game which has also been proposed in \cite{GWZ18}. Let $D_1$ and $D_2$ be two integral domains both contained in some larger domain $D$. Now, our players Nora and Wanda are required to choose coef{f}icients from $D_1$ and the winner is decided according as whether the polynomial has a root in $D_2$ or not. We provide one such family of examples:
  \begin{theorem}\label{Th:Secondvariant}
    Let $d > 1,\ p$ be any prime integer and the players be choosing coef{f}icients from $D_1 = \Q$. Then, the last player can win if roots are to be avoided/demanded in any subring $D_2 \subseteq \Q_p$, the f{i}eld of $p$-adic numbers.
  \end{theorem}
  \noindent As any unital subring of $\Q_p$ will contain all the integers and Wanda's strategy as the last player is guaranteed to produce an integral root (cf.~Lemma~\ref{L:GWZ}), her win is secured. The rest follows as a consequence of the following:
  \begin{theorem}\label{Th:QuotFieldDVR}
    Let $d > 1$ and $R = \operatorname{Frac} (\O)$ be the field of quotients of a complete discrete valuation ring \O\ with its residue f{i}eld having f{i}nite cardinality. Then, whosoever plays the last move of the game has a winning strategy.
  \end{theorem}
  \noindent A proof of Theorem~\ref{Th:QuotFieldDVR} for $d \neq 3$ is presented in \S\,\ref{S:ult} in the form of Lemma~\ref{L:quad}, Proposi{t}ion~\ref{P:biquad} and Theorem~\ref{Th:dgeq5}. F{i}niteness of the residue f{i}eld is required only for $d = 3$ in Proposi{t}ion~\ref{P:cub} which can be circumvented when $\operatorname{char}\,(R)$ equals $3$.\\[-0.1cm]

  Af{t}er this, Theorem~\ref{Th:Secondvariant} can be deduced for all but cubic polynomials from Corollary~\ref{C:unr} and the following discussion. The theory of Newton polygons helps us to achieve our goal in Corollary~\ref{C:cub}. In the last sect{i}on, we have an applicat{i}on of our results on discrete valuat{i}on rings to the set{t}ting where $D_1$ is the f{i}eld of rat{i}onal numbers and $D_2$ is the union of all f{i}nite abelian extensions of \Q. The lat{t}er is the same as the union of all cyclotomic extensions of \Q\ by the Kronecker-Weber Theorem~\cite[\S\,14.5]{DF04}.
  \begin{theorem}
    Let the degree $d$ be either equal to $3$ or greater than $8$. If both Nora and Wanda are required to choose polynomial coef{f}icients from \Q, then the last player can win if roots are to be avoided/demanded in any subring $D_2$ of the maximal abelian extension of \Q.
  \end{theorem}
  
\section{\protect{Preliminary observat{i}ons}}
This small sect{i}on records two lemmata which hold in the maximum generality and are used in both of the sect{i}ons that follow. We denote the polynomial obtained at the end of the game by $f$ and its coefficients by $a_i$'s, i.~e.,
\[
f(x)=a_d x^d + \cdots + a_{i }x^{i } + \cdots + a_0.
\]
To begin with, there is the case of linear polynomials.
  \begin{lemma}\label{L:lin}
    If $d = 1$ and $R$ is a unital ring, Wanda can always win.
  \end{lemma}
  \begin{proof}
    If she plays f{i}rst, it suf{f}ices to choose $a_1 = 1$ and if she plays second, the choice of $a_1$ (or $a_0$) has to be the same as Nora's pick for $a_0$ (or $a_1$). In the lat{t}er situat{i}on, $-1$ is a root of the linear polynomial.
  \end{proof}
  The statement below is essent{i}ally available in~\cite{GWZ18} but we present the proof here for the sake of clarity and completeness.
  \begin{lemma}\label{L:GWZ}
    If $R$ is a commutat{i}ve ring with unity and Wanda makes the last move of the game, then she can win.
  \end{lemma}
  \begin{proof}
    For $d \geq 4$, Wanda can make sure that either she or Nora has chosen both $a_d$ and $a_0$ before the last move. Then,
    \[
      f(x) = g(x) + a_i x^i
    \]
    for some $i \notin \{ 0, d \}$, f{i}xed $g$ and $a_i$ yet to be determined by Wanda. She lets $a_i = - g(1)$ and wins with $1$ being a root of $f$ in $R$.\\[-0.2cm]
    
    For $d = 3$, Nora has to be Player\,\Rmnum{1} if Wanda has to play last. Whenever Nora chooses $a_0$ or $a_3$, Wanda picks the other to be the same. This strategy is also employed for $a_1$ and $a_2$ so that $-1$ is a root of the f{i}nal polynomial.\\[-0.2cm]
    
    When $d = 2$, Wanda plays the f{i}rst move as well. She lets $a_1 = 0$ and later picks $a_0$ (or $a_2$) to be $-a_2$ (or $-a_0$) so that $\pm 1$ are roots of our polynomial.\\[-0.2cm]
    
    We have already set{t}led the mat{t}er of linear polynomials in Lemma~\ref{L:lin}.
  \end{proof}

\section{F{i}nite cyclic rings}
  In this sect{i}on, $R$ will always be a f{i}nite cyclic ring $\Z/N\Z$ but not a f{i}eld. This, in part{i}cular, implies that $N$ is not a prime. We first establish the advantage possessed by the last mover for even degree polynomials.
  \begin{lemma}\label{L:even}
    Let $d > 1$ be even and $R = \Z / N\Z$. Then, whosoever plays last has a winning strategy.
  \end{lemma}
  \begin{proof}
    If Wanda makes the last move, she can win by Lemma~\ref{L:GWZ}. Now, consider the other scenario. As $d$ is even, the total number of coef{f}icients to be chosen is odd. Nora has to be Player\,\Rmnum{1} too if she has to play the last move. She chooses $a_0 = 1$ so that
    \[
      f(x) = xg(x) + 1
    \]
    and any non-unit cannot be a root of the polynomial obtained at the end of the game. On her last move, Nora faces $f(x) = h(x) + a_ix^i,\ i \neq 0$ with $h$ f{i}xed and $a_i$ to be chosen by her next. There are $\varphi (N)$ many units in $\Z / N\Z$ and she should only avoid choosing from the set
    \[
      \left\{ - u^{-i}h(u)\pmod{N} \mid u \in R^* \right\} \cup \{ 0 \}
    \]
    which has cardinality at most $\varphi(N) + 1 < N$, as we assumed $N$ is not a prime number. Her requirements for $a_i$ are then evidently feasible.
  \end{proof}
  We are lef{t} to study the case when $d > 1$ is odd and Nora is the last player. This is examined in several steps depending on the prime-factorizat{i}on of $N$.
  \begin{lemma}\label{L:psq}
    If $N = p^2$ for some prime $p$ and $d \geq 2$, then the last player can win.
  \end{lemma}
  \begin{proof}
    Wanda can win if she is the last player by following the strategy in Lemma~\ref{L:GWZ}. Nora can win using Lemma~\ref{L:even} if $d$ is even and she is the last player. Therefore, let us assume that $d$ is odd and Nora plays last. This also means that Wanda is Player\,\Rmnum{1} for us.\\[-0.2cm]
    
    If Wanda's f{i}rst move is to choose some $a_i$ for $i \neq 0$, Nora immediately picks $a_0 = 1$ next. By the same reasoning as in the proof of Lemma~\ref{L:even}, Nora can finish the game of{f} with a polynomial which has no roots in $R$. This argument also works when Wanda chooses $a_0$ from $R^*$ on her f{i}rst move.\\[-0.2cm]
    
    Let $a_0 = u_0p$ be Wanda's f{i}rst pick for some representat{i}ve $u_0 \in R^*$. This leads to Nora f{i}xing $a_1 = 0$ following which no mult{i}ple of $p$ can be a root of $f$ in $\Z / p^2\Z$. Her choice is legal as $d > 1$. On her last move, she has to avoid at most $\varphi(N) + 1 < N$ many values of $a_i$ corresponding to $\varphi(N)$ many units in $R$ in addit{i}on to the zero element. Nora does so and wins the game.
  \end{proof}
  One may extend this much further as shown below.
  \begin{lemma}\label{L:cfree}
    The last-mover advantage holds for cube-free numbers and $d \geq 2$.
  \end{lemma}
  \begin{proof}
    In view of prior observat{i}ons, we may restrict ourselves to when
    \begin{enumerate}
      \item $d > 1$ is odd,
      \item $N$ is not a prime power,
      \item Wanda is Player\,\Rmnum{1}, and
      \item her f{i}rst choice is $a_0 \in R \setminus R^*$.
    \end{enumerate}
    Nora tries to f{i}nd a prime $p$ such that $N = p^{k}q$ with $k \in \{ 1, 2 \},\ ( p, q ) = 1$ and $p$ does not divide $a_0$. If her search is successful, she chooses $a_d = 1$. On her last move, Nora chooses the value of $a_i$ for which $f\pmod{p}$ has no roots in $\F_p$. This is realizable because $f\pmod{p}$ cannot have $0$ as a root in $\F_{p}$ which means that Nora has to avoid at most $p - 1$ classes of $a_i$ modulo $p\ (i \neq 0, d )$ corresponding to as many classes in $\F_{p}^*$. Thereaf{t}er, $f$ cannot have a root in $R$ since $f\pmod{p}$ does not have a root in the quot{i}ent ring $\Z / p\Z$.\\[-0.2cm]
    
    The search for a suitable $p$ as above may fail only when there exists a prime $p$ dividing $a_0$ such that $p^2$ divides $N$ but not $a_0$. In this case, Nora lets $a_1 = 0$ (allowed since $d > 1$) and plays arbitrarily till before her last move. At that stage, she reduces $f$ modulo $p^2$ and chooses the remaining coef{f}icient $a_i$ so that $f\pmod{p^2}$ does not have a root in $\Z/p^2\Z$. This is possible as she has to avoid at most $\varphi(p^2) + 1 < p^2$ many equivalence classes of $a_i\pmod{p^2}$ coming from $\varphi(p^2)$ many units in $\Z/p^2\Z$ and the fact that $a_d$ is not allowed to be zero. Then, $f$ cannot have a root in $R$.
  \end{proof}
  The previous lemmata build upon the ideas of \citeauthor*{GWZ18} for f{i}nite f{i}elds $\F_p$. We can avoid dealing with permutat{i}on polynomials here because Wanda has to necessarily choose the constant term on her f{i}rst move if she wants to win for odd degree polynomials. The results so far may suggest to the reader that for non-linear polynomials, the last mover in the coef{f}icient-choosing game has an advantage over rings with zero divisors. Our lemma below shows any such intui{t}ion to be false.
  \begin{lemma}\label{L:oddpow}
    If $N = p^{2k + 1}$ for some prime $p,\ k \in \N^*$ and $d$ is odd, then Wanda always has a winning strategy.
  \end{lemma}
  \begin{proof}
    In the light of Lemmata~\ref{L:lin} and \ref{L:GWZ}, we only need to examine the case when $d > 1$ and Wanda is Player\,\Rmnum{1}. She begins by choosing $a_0 = -p^{2k}$.\\[-0.2cm]
    
    If Nora picks a coef{f}icient other than $a_1$, Wanda sets $a_1 = 1$ so that
    \[
      f(x) = x^2g(x) + x - p^{2k}
    \]
    and $x = p^{2k}$ will be a root of $f$ in $R$. If Nora's f{i}rst move is $a_1 = u_1p^i$ for some choice of representat{i}ve $u_1 \in R^*$ and $i < k$, we similarly have
    \[
      f(x) = h(x) + u_1p^ix - p^{2k}
    \]
    whence $u^{-1}_1p^{2k - i}$ is a root of $f$ in $R$ and Wanda is dest{i}ned to be the winner. This is because for $n \geq 2$, we get $n ( 2k - i ) - 2k > 0$ and all higher degree terms const{i}tut{i}ng $h(x)$ are automat{i}cally zero for $x$ which are mult{i}ples of $p^{2k - i}$.\\[-0.2cm]
    
    Next, let $a_1 = u_1p^k$ be Nora's f{i}rst move for some choice of representat{i}ve $u_1 \in R^*$. Wanda lets $a_2 = 0$ (allowed as $d \neq 2$) so that for $n > 2$,
    \[
      nk = 2k + ( n - 2 ) k \geq 2k + 1
    \]
    making $x = u_1^{-1}p^k$ to be a root of any $f$ obtained af{t}erwards.\\[-0.2cm]
    
    If Nora chooses $a_1 = b_1p^{k + 1}$ for some $b_1 \in R$, Wanda simply lets $a_2 = 1$ leading to $f(x) = x^3g(x) + x^2 + b_1p^{k + 1}x - p^{2k}$. Clearly, $\pm p^k$ are roots of $f$ in $\Z/N\Z$.
  \end{proof}
  The same is also true for all but two even prime powers.
  \begin{lemma}\label{L:evnpow}
    If $N = p^{2k}$ for some prime $p,\ k \geq 3$ and $d$ is odd, then Wanda can always win.
  \end{lemma}
  \begin{proof}
    We concentrate on the case when $d > 1$ and Wanda is Player\,\Rmnum{1}. She begins by choosing $a_0 = -p^{2k - 1}$.\\[-0.2cm]
    
    If Nora picks a coef{f}icient other than $a_1$, Wanda sets $a_1 = 1$ so that
    \[
      f(x) = x^2g(x) + x - p^{2k - 1}
    \]
    and $x = p^{2k - 1}$ will be a root of $f$ in $R$ following from $( p^{2k - 1} )^2 \equiv 0\pmod{p^{2k}}$. If Nora's f{i}rst move is $a_1 = u_1p^i$ for some choice of representat{i}ve $u_1 \in R^*$ and $i < k$, we have $f(x) = h(x) + u_1p^ix - p^{2k - 1}$ whence $u^{-1}_1p^{2k - 1 - i}$ is a root of $f$ in $R$ and Wanda will be the winner. This is because for $n \geq 2$,
    \[
      n ( 2k - 1 - i ) = 2k + 2 ( n - 1 ) k - n ( i + 1 ) \geq 2k + ( n - 2 ) k \geq 2k
    \]
    and all higher degree terms in $h(x)$ are zero for mult{i}ples of $p^{2k - 1 - i}$.\\[-0.2cm]
    
    Next, let $a_1 = u_1p^k$ be Nora's f{i}rst move for some choice of representat{i}ve $u_1 \in R^*$. Wanda lets $a_2 = 0$ so that for $n > 2$,
    \[
      n ( k - 1 ) = 2k + ( nk - 2k - n )
    \]
    making the term inside the parentheses on the right side to be non-negat{i}ve for $k \geq 3$ and $x = u_1^{-1}p^{k - 1}$ to be a root of any $f$ obtained later.\\[-0.2cm]
    
    If Nora chooses $a_1 = b_1p^{k + 1}$ for some $b_1 \in R$, Wanda simply lets $a_2 = p$ leading to $f(x) = x^3g(x) + px^2 + b_1p^{k + 1}x - p^{2k - 1}$. Clearly, $\pm p^{k - 1}$ are roots of $f$ in $\Z/N\Z$ for the same reason as given in the previous paragraph.
  \end{proof}
  An addi{t}ional feature of the last two results is that if $d$ is greater than one and odd, Nora may even be allowed to choose the leading coe{f}ficient $a_d$ to be $0$.
  \begin{corollary}\label{C:CRT}
    Let $d$ be odd and $N = p^{k} N_2$ for some prime $p,\ k \in \{ 3, 5, 6, 7, \ldots \}$ and $p$ not dividing $N_2$. Then, Wanda can always win.
  \end{corollary}
  \begin{proof}
    We may just focus on $d > 1$ and Wanda being Player\,\Rmnum{1} again. Also, recall the ring isomorphism
    \begin{equation}\label{E:ri}
      \Z / N\Z\ \simeq\ ( \Z / p^k\Z ) \times ( \Z / N_2\Z )
    \end{equation}
    using the Chinese Remainder Theorem. By Lemmata~\ref{L:oddpow} and \ref{L:evnpow}, Wanda has a winning strategy beginning with a choice of the constant term in $\Z / p^k\Z$. If $a_{0, 1} \in \Z/p^k\Z$ denotes such a choice for Wanda, she chooses $a_0 \in \Z / N\Z$ for which $a_0 \equiv a_{0, 1}$ modulo $p^k$ while $a_0 \equiv 0$ modulo $N_2$. For all of Nora's subsequent moves, Wanda reduces the coef{f}icients modulo $p^k$ and computes her response over $\Z / p^k\Z$. If $a_{i, 1}$ is part of Wanda's winning strategy there, she always picks $a_i \in R$ for which $a_i \equiv a_{i, 1}\pmod{p^k}$ and $a_i \equiv 0\pmod{N_2}$.\\[-0.2cm]
    
    We denote $x_1 \in \Z / p^k\Z$ to be a root of $f\mod{p^k}$ in the quot{i}ent ring $\Z / p^k\Z$, where $f$ is the polynomial obtained at the end. Then, the element $x_0 \in R$ such that $x_0 \equiv x_1$ modulo $p^k$ and $x \equiv 0$ modulo $N_2$ will be a root of $f$ in $R$ by the ring isomorphism~\eqref{E:ri}.
  \end{proof}
 For a f{i}xed integer polynomial $f \in \Z[x]$ and a prime number $p$, the problem of count{i}ng the number of its roots modulo the prime power $p^k$ seems to be very challenging and having myriad applicat{i}ons. The best known determinist{i}c algorithm has a t{i}me complexity exponent{i}al in $k$. We point to~\cite{KRRZ} and the references therein for more on this. A Las Vegas randomized algorithm for comput{i}ng the number of roots in the ring $\Z / p^k\Z$ is also given over there which takes t{i}me less than some polynomial in terms of $k$.

\subsection{\protect{The curious case of fourth prime powers}}\label{S:biquad}
  The discussion so far tells us that it remains to analyse the outcome of the game played for non-linear polynomials over quot{i}ents of \Z\ by fourth powers of prime numbers and mult{i}ples thereof. The arguments are only slightly di{f}ferent for even and odd primes. We present the one for odd primes f{i}rst, followed by a proof for $p = 2$.
  \begin{lemma}\label{L:podd}
    Let $d$ be odd and $N = p^4$ for some prime $p \neq 2$. Then, Wanda always has a winning strategy.
  \end{lemma}
  \begin{proof}
    We may assume $d > 1$. If Wanda is Player\,\Rmnum{1}, she chooses $a_0 = -p^2$. It is enough for her to target the cubic part of the polynomial to be zero modulo $p^4$ for some mult{i}ple of $p$, say $up$, for then the contribut{i}ons from higher terms will each be automat{i}cally zero. 
    \begin{itemize}
      \item If Nora doesn't f{i}x $a_1$ next, Wanda lets it to be equal to $1$ so that $p^2$ is a root of $f$. If Nora does choose $a_1$ to be a unit in $\Z/N\Z$, then $a_1^{-1}p^2$ is a root of our polynomial.
      \item \underline{$\mathbf{a_1 = u_1p}$} : If $a_1 = u_1p$ is Nora's f{i}rst choice for some $u_1 \in ( \Z/N\Z )^*$, Wanda lets $a_2 = 0$. Note that the choice of the representat{i}ve $u_1$ is specif{i}ed up to an addi{t}ive factor of $kp^3$ for some $k \in R$. The ef{f}ect{i}ve port{i}on of the polynomial from Wanda's perspect{i}ve, evaluated at some $x = up$, is
      \begin{equation}\label{E:up}
        a_3\cdot(up)^3 + u_1p\cdot(up) - p^2\ \equiv\ ( a_3u^3p + u_1u - 1 )p^2.
      \end{equation}
      Given that $u_1\pmod{p^2}$ is well-def{i}ned in the quot{i}ent ring $\Z / p^2\Z$ and independent of our choice of representat{i}ve, we let $u_1^{-1}$ denote its inverse in $R/p^2R$. Take one of the shif{t}ed terms $u = u^{-1}_1 + kp$ for all of whom
      \[
        a_3(u^{-1}_1 + kp )^3p\ \equiv\ a_3u_1^{-3}p\pmod{p^2}
      \]
      while $u_1 ( u_1^{-1} + kp ) - 1 \equiv u_1kp$ modulo $p^2$. No mat{t}er what choice of $a_3$ Nora (or Wanda) make, one of these $k$'s will help Wanda to make the term within parentheses on the right side of~\eqref{E:up} to be zero modulo $p^2$. Our polynomial evaluated at such a $up$ will vanish as a consequence.
      
      \item \underline{$\mathbf{a_1 = u_1p^2}$} : If Nora's f{i}rst choice is $a_1 = u_1p^2$ for some unit $u_1$ specif{i}ed up to an addi{t}ive factor of $kp^2$, Wanda lets $a_2 = 1$ so as to face
      \[
        a_3\cdot(up)^3 + (up)^2 +u_1p^2\cdot(up) - p^2\ \equiv\ \big( u (a_3u^2 + u_1 )p + u^2 - 1 \big) p^2\pmod{p^4}.
      \]
      For $u = 1 + kp$, we have $u ( a_3u^2 + u_1 )p \equiv ( a_3 + u_1 )p$ and $u^2 - 1 \equiv 2kp$ modulo $p^2$. As there is an assurance that $p \neq 2$, an appropriate value of $k$ will help Wanda to f{i}nd a root of $f$ in $R$ which is a mult{i}ple of $p$.
      
      \item \underline{$\mathbf{a_1 = b_1p^3}$} : If Nora lets $a_1$ be a mult{i}ple of $p^3$ on her f{i}rst move, Wanda chooses $a_2 = 1$ so that we have a similar situat{i}on as for $a_1 = u_1p^2$ above.
    \end{itemize}
    As we have exhausted all of Nora's possible opt{i}ons, this f{i}nishes the proof.
  \end{proof}
  It should be remarked here that not only can Wanda ensure the f{i}nal polynomial to have roots, she can force those roots to lie in $pR \subset R$ provided $p$ is odd. When $p$ equals $2$, the game is t{i}lted towards her for small values of $d$ only.
    \begin{lemma}\label{L:peven}
    Let $d = 3$ as before while $R = \Z/16\Z$. Then, Wanda can win irrespect{i}ve of being the f{i}rst or the second player.
  \end{lemma}
  \begin{proof}
    When Wanda is Player\,\Rmnum{1}, she begins by choosing $a_0 = 12 \equiv - 2^2\pmod{16}$. As seen in the proof of Lemma~\ref{L:podd}, Nora is compelled to pick $a_1$ next if she wants to have a chance at winning and it should be a non-unit.
    \begin{itemize}
      \item \underline{$\mathbf{a_1 = 2u_1}$} : If this is Nora's f{i}rst move for some choice of representat{i}ve $u_1 \in ( \Z/16\Z )^*$,
      Wanda can cont{i}nue to follow her strategy as in Lemma~\ref{L:podd} when Nora had chosen $a_1$ to be a unit mult{i}ple of $p$ in $\Z / p^4\Z$. The reader may verify for herself that nothing there prevents Wanda's victory if we take $p$ to be equal to $2$.

      \item \underline{$\mathbf{a_1 = 4u_1}$} : Suppose Nora makes such a choice for some $u_1 \in ( \Z/16\Z )^*$.
      On facing this move from Nora, Wanda picks $a_2 = 1 - 2u_1 \in R^*$. Admi{t}tedly, there is a discret{i}on involved here about the representat{i}ve for $u_1$. It can be checked that Wanda can choose any one of them.\\[-0.3cm]
      
      Now, Nora cannot allow $a_3$ to belong to the set
      \begin{equation}\label{E:units}
        \{ - ( a_2 u + 4u_1 u^2 - 4 u^3 ) \mid u \in R^* \}
      \end{equation}
      or else, the corresponding $1/u$ will be a root of $f$. As Wanda took $a_2$ to be a unit, all these elements in~\eqref{E:units} are units of $R$ too. We claim that for any f{i}xed $u_1$ and $a_2$ as above, this set const{i}tutes $R^*$. Suppose not and let
      \[
        a_2u + 4u_1u^2 - 4u^3 \equiv a_2v + 4u_1v^2 - 4v^3\pmod{16}
      \]
      for some $u, v \in R^*$. On rearranging, we see that $u - v$ has to be a mult{i}ple of $4$ since $a_2$ is not. This makes the right side of the congruence
      \[
        a_2 ( u - v ) \equiv -4u_1 ( u^2 - v^2 ) + 4 ( u^3 - v^3)\pmod{16}
      \]
      to be zero modulo $16$ which in turn requires $u - v$ to be same as well. In part{i}cular, it implies that Nora cannot choose $a_3$ to be a unit in $R$. If $a_3$ is even instead, we will have $f(2) = 0$ in $R$.
      
      \item \underline{$\mathbf{a_1 = 8b_1}$} : If Nora lets $a_1$ be a mult{i}ple of $8$, Wanda chooses $a_2 = 1$ so that Nora can't pick $a_3$ to be odd again (recall the reasoning for $a_1 = 4u_1$) and $\pm 2$ are roots of $f$ otherwise.
    \end{itemize}
    Our claim has been established.
  \end{proof}
  The requirement that the `leading' and the constant coef{f}icients of the polynomial be non-zero is an art{i}f{i}cial technicality of the game introduced to remove redundancies. It would, however, have made no di{f}ference to either players' fortunes in Lemmata~\ref{L:podd} and \ref{L:peven} even if we had allowed Nora the freedom to choose $a_d = 0$ if she wished so. Af{t}er this brief remark, Corollary~\ref{C:CRT} can be strengthened to say that
  \begin{corollary}\label{C:Wanda}
    Let $d$ be odd, $N$ not be cube-free and at least one of the following hold:
    \begin{enumerate}
        \item $d = 3$,
        \item $N = 8N_2$ for some posi{t}ive integer $N_2 \not\equiv 2\pmod{4}$, or
        \item there exists an odd prime $p$ such that $p^3$ divides $N$.
    \end{enumerate}
    Then, Wanda always has a winning strategy.
  \end{corollary}
  \begin{proof}
    Denote $N = p^3N_2$ for some prime $p$ and $N_2 \in \N^*$. We are done by Corollary~\ref{C:CRT} if $p$ does not divide $N_2$ or if the mult{i}plicity of $p$ in $N_2$ is more than $1$. Else, by Lemmata~\ref{L:podd} and \ref{L:peven}, Wanda has a winning strategy beginning with a choice of the constant term in $\Z / p^4\Z$ (here, $d=3$ if $p=2$). As in Corollary~\ref{C:CRT}, this observation completes the proof.
  \end{proof}
  Together with Lemma~\ref{L:cfree}, this completes the picture for cubic polynomials. We need to work a li{t}tle bit more for other higher odd degrees.
  \begin{lemma}\label{L:4plus}
    Let $d > 3$ be odd. Then, whosoever plays last can win over $\Z/16\Z$.
  \end{lemma}
  \begin{proof}
    One has to only invest{i}gate what happens when Wanda is Player\,\Rmnum{1} for else, she can certainly win. It has been explained before that Wanda has to choose $a_0$ to be a non-unit on her f{i}rst move provided her desire to be the winner. If this $a_0 = 2u_0$ for some choice of representat{i}ve $u_0 \in R^*$, Nora can play the strategy for $f\pmod{4}$ to not have any roots in $\Z / 4\Z$. We now examine
    \begin{itemize}
      \item \underline{$\mathbf{a_0 = 4u_0}$} : Let this be Wanda's choice for some choice of representat{i}ve $u_0\in R^*$. In this case, Nora picks $a_1 = 8$ immediately af{t}er. If Wanda doesn't f{i}x $a_2$ next, Nora may let it be equal to zero at her second move. This ensures that any mult{i}ple of $2$ cannot be a root of $f$ obtained in the end as all but the constant term of the polynomial are divisible by $2^3$ for $x = 2b$. In other words, Nora has to worry about elements of $R^*$ alone for her last move implying that she can be the winner.\\[-0.3cm]
      
      Suppose Wanda does choose some $a_2$ on her second move. As $a_1 \equiv 2^3$ and $a_0 \not\equiv 0$, any mult{i}ple of $4$ cannot be a root of $f$. If Nora wants to eliminate the possibility that $2u\ (u\in R^*)$ is a root, then she needs
      \[
        (2u)^4\cdot h(2u) + a_3\cdot(2u)^3 + a_2\cdot(2u)^2 + 2^3\cdot 2u + 4u_0 \not\equiv 0\pmod{16}
      \]
      which is the same as saying that
      \[
        a_3\cdot(2u)^3 + a_2\cdot(2u)^2 + 4u_0 \not\equiv 0\pmod{16}
      \]
      or equivalently, Nora wants
      \[
        8a_3\ \not\equiv\ - 4u ( a_2 + u_0u^2 )\pmod{16}
      \]
      for all $u \in R^*$. We remind the reader that $u_0$ is well-def{i}ned in the quot{i}ent $\Z / 4\Z$ and as is $a_2$. When the lat{t}er is even, any value of $a_3$ will do. Else, the map $u \mapsto -u (a_2 + u_0u^2 ) \pmod{4}$ from $R^*$ to $\Z / 4\Z$ is constant for any f{ixed} $a_2, u_0 \in ( \Z / 4\Z )^*$. Nora can, therefore, declare $a_3$ to be a di{f}ferent mult{i}ple of $8$ and rule out $2u$ from being roots of $f$. She takes care of the unit elements of $R$ on her last move.
      
      \item \underline{$\mathbf{a_0 \equiv 8\pmod{16}}$} : Then, Nora chooses $a_1 = 4$. If Wanda doesn't choose $a_2$ immediately af{t}er that, Nora can let it be equal to $1$ at her next move ruling out all odd mult{i}ples of $2$ from being roots of $f$. This is because otherwise there will exactly one term of the polynomial which is not divisible by $8$. Even mult{i}ples of $2$ cannot be roots of such an $f$ anyway. Our arguments also hold if Wanda chooses $a_2$ from $R^*$.\\[-0.3cm]
      
      Let us assume that Wanda picks $a_2 = 2b_2$. Nora would like to have
      \[
        a_3\cdot(2u)^3 + 2b_2\cdot(2u)^2 + 4\cdot(2u) + 8 \not\equiv 0\pmod{16}
      \]
      for all $u \in R^*$ which is possible if{f}
      \[
        a_3 \not\equiv - ( b_2u + u^2 + u^3 )\pmod{2}
      \]
      for `all' units in $\Z / 2\Z$. As $b_2$ has already been f{i}xed before, she can choose an $a_3$ as required. The units of $\Z/16\Z$ are prevented from being roots of $f$ on Nora's last turn.
    \end{itemize}
    Since we covered all of Wanda's opt{i}ons, the proof is done.
  \end{proof}
  The crucial di{f}ference between Lemma~\ref{L:peven} and \ref{L:4plus} is that for $d = 3$, the leading coef{f}icient $a_3$ can't be zero and Nora has to simultaneously stop all elements of $R$ from being roots of $f$ when choosing $a_3$. For odd $d > 3$, she has adequate freedom to handle even integers f{i}rst and worry about the units later.
  \begin{corollary}\label{C:last}
    If $d > 3$ and $N = 16N_2$ where $N_2$ is a cube-free odd integer, then the last player has a winning strategy.
  \end{corollary}
  \begin{proof}
    We conf{i}ne ourselves to when Wanda is Player\,\Rmnum{1} and $d$ is odd so that Nora is the last player. Wanda has to necessarily choose $a_0$ to be a non-unit on her f{i}rst move itself if she would like to win. At this stage, Nora f{i}nds a prime $p$ such that $N = p^{\eps}q$ with $( p, q ) = 1$, $a_0 \not\equiv 0\pmod{p^{\eps}}$ and $\eps$ having a posi{t}ive value in accordance with the statement of the claim. She then plays the strategy for $f\pmod{p^{\eps}} \in ( \Z/p^{\eps}\Z ) [x]$ to not have any roots in $\Z/p^{\eps}\Z$.
  \end{proof}
  
\section{\protect{Quot{i}ent f{i}eld of a complete DVR}}\label{S:ult}
  Let \O\ be a complete discrete valuat{i}on ring with \K\ being its f{i}eld of quot{i}ents. We denote a generator of the maximal ideal $\mathfrak{p} \subsetneq \O$ by $\rho$. For example, $K$ can be a f{i}nite extension of $\Q_p$ or $\F (( T ))$, the f{i}eld of formal Laurent series in a transcendental variable $T$ with coef{f}icients coming from some `f{i}eld of constant' \F. In the lat{t}er set{t}ing, the variable $T$ plays the role of a uniformiser.\\[-0.1cm]
  
  If $\O^* = \O \setminus \mathfrak{p}$ is the mult{i}plicat{i}ve group of units of \O, we say that the integer $n$ is the \emph{order} of a non-zero element $\alpha$ when $\alpha \in \rho^n\O^*$. It is then extended to the whole of \K\ by def{i}ning $\ord (0) = + \infty$. A unique and well-def{i}ned extension of the \ord\ funct{i}on is also possible for the algebraic closure of \K. The f{i}eld $\overline{\K}^{\text{alg}}$ is endowed with a norm given by
  \[
    |\alpha| := e^{-\ord \alpha}
  \]
  which helps us to have a not{i}on of distance on such f{i}elds. It turns out to be ultrametric in nature. The game for linear polynomials over $K$ has been dealt with in Lemma~\ref{L:lin} and this mot{i}vates us to go further.
  \begin{lemma}\label{L:quad}
    For $d = 2$, the last player has a winning strategy over \K.
  \end{lemma}
  \begin{proof}
    Let us focus on Nora being the f{i}rst as well as the last player simultaneously. She begins by choosing $a_1 = 0$. Af{t}er Wanda has chosen any non-zero $a_0$ (or $a_2$), Nora simply picks $a_2$ (or $a_0$) to be such that $\ord a_2 \not\equiv \ord a_0\pmod{2}$. Clearly, many such admissible choices are available to her and it guarantees that the quadrat{i}c polynomial $a_2x^2 + a_0$ has no root in $K$ owing to order considerat{i}ons of the two contribut{i}ng terms.
  \end{proof}
  It is to be noted for future purposes that Nora can addi{t}ionally have all of her choices during the proof of Lemma~\ref{L:quad}, including the last one, to be rat{i}onal numbers when $K = \Q_p$. She may similarly choose all her $a_i$'s to belong to $\F_p ((T))$ when $K = \F((T))$ for some $\F \subseteq \overline{\F_p}^{\text{alg}}$, if she desires so. Nora's strategy also succeeds when the polynomial coef{f}icients are to be chosen from the discrete valuat{i}on ring \O\ by both the players. On the other hand, Wanda's winning strategy over \O\ and as the last player is still governed by Lemma~\ref{L:GWZ}. We now skip the case of $d = 3$ for a moment in pursuit of other higher goals.
  \begin{prop}\label{P:biquad}
    For $d = 4$, Nora can win over $K$ if she is the last player.
  \end{prop}
  \begin{proof}
    Nora gets two moves before her last one. Hence, she can ensure that either she gets to choose both $a_1$ and $a_3$ to be zero in her f{i}rst two turns or the last coef{f}icient to be determined by her is one of $a_0$ or $a_4$. If it is $a_0$, Nora wants
    \begin{equation}\label{E:biquad}
      a_0 \neq - ( a_4x^4 + \cdots + a_1x )
    \end{equation}
    for all $x$ in $K^*$. Let
    \[
      M_1 := \min \big\{ \ord a_3 - \ord a_4,\ \frac{1}{2} ( \ord a_2 - \ord a_4 ),\ \frac{1}{3} ( \ord a_1 - \ord a_4 ) \big\},
    \]
    with the const{i}tuent terms ignored if the corresponding $a_i = 0$, and let $x$ have order less than $M_1$. When $a_1 = a_2 = a_3 = 0$, we set $M_1 = \infty$ and $x$ can be any non-zero element. Then, the right side of~\eqref{E:biquad} has its order belonging to the set $\{ \ord a_4 + 4n \mid n < M_1 \}$. For all other $x \neq 0$, the orders are bounded from below by
    \[
      M_2 := \min\,\{\,\ord a_i + iM_1 \mid i = 1, \ldots, 4\,\}.
    \]
    Nora may choose any f{i}eld element, even a rat{i}onal number or an element of $\F_p((T))$ as the case may be, whose order doesn't lie in the union
    \[
      \{\,\ord a_4 + 4n \mid n < M_1\,\}\ \cup\ \{\,n \geq M_2\,\}
    \]
    and win the game.\\[-0.2cm]
    
    If $a_4$ was lef{t} for her to decide at the last move, we can take the related polynomial $g \in \K[x]$ such that $f(x) = x^4g(1/x)$ for all non-zero $x$ and $a_4$ becomes the constant term of $g$. Since $a_0a_4 \neq 0$, neither of $f$ or $g$ can have zero as a root and $f(x) = 0$ if{f} $g(1/x) = 0$ for $x \in \K^*$. This borrowed trick from~\citep{GWZ18} reduces the problem to the previous scenario discussed.\\[-0.2cm]
    
    F{i}nally, there is exactly one of the two possibili{t}ies when the last coef{f}icient to be chosen is $a_2$. Either $\ord a_4 \equiv \ord a_0\pmod{2}$ or not. At any rate, it is implicit that Nora had eliminated $a_1$ and $a_3$ earlier. She would now like to have
    \begin{equation}\label{E:par}
      a_2 \neq - ( a_4x^2 + a_0x^{-2} )\quad\forall x \in K^*.
    \end{equation}
    Nora declares $a_2 = 0$ if the pari{t}ies of $\ord a_0$ and $\ord a_4$ are di{f}ferent modulo $2$.\\[-0.2cm]
    
    Else, the orders of all the non-zero f{i}eld elements given by the right side of~\eqref{E:par} have same parity modulo $2$ as $\ord a_0$ and $\ord a_4$ for $x \in K^*$ such that $\ord x \neq ( \ord a_0 - \ord a_4 ) / 4$. For any $x$ with $\ord x = ( \ord a_0 - \ord a_4 ) / 4$, the order of the expression on the right side of~\eqref{E:par} is at least $( \ord a_0 + \ord a_4 ) / 2$. Nora may choose $a_2$ with its order in the complementary subset
    \[
      \{\,\ord a_0 + 2n + 1 \mid n \in \Z\,\}\ \cap\ \{\,n < ( \ord a_0 + \ord a_4 ) / 2\,\}.
    \]
    She may furthermore have such an $a_2$ to be a rat{i}onal number or belong to $\F_p((T))$ depending on whether $\K = \Q_p$ or $\K \supset \F_p ((T))$, respect{i}vely.
  \end{proof}
  The above proof captures all the complexit{i}es which may arise for the remaining higher degrees and then some. We proceed without further ado.
  \begin{theorem}\label{Th:dgeq5}
    For $d > 4$, the last player can win over $K$.
  \end{theorem}
  \begin{proof}
    Let us concentrate solely on Nora being the last player. As $d$ is at least $5$, she must have got at least two chances before her last move. Nora makes sure that both $a_1$ and $a_{d - 1}$ have been chosen at the end of her second turn.\\[-0.2cm]
    
    \noindent Thereaf{t}er, if $a_0$ is the last coef{f}icient to be decided by Nora, she wants
    \begin{equation}\label{E:a0}
      a_0 \neq - x^d \big( a_d + a_{d - 1}x^{-1} + \cdots + a_1 x^{-(d - 1 )} \big)\quad\forall x \in K^*.
    \end{equation}
    Denote $M_1 := \min\,\{\,( \ord a_i - \ord a_d ) / ( d - i ) \mid 0 < i < d,\ a_i \neq 0\,\}$ with $M_1$ def{i}ned to be $+\infty$ if the minimum is to be taken over an empty set. For a non-zero $x$ with $\ord x = n < M_1$, the right side of~\eqref{E:a0} has order equal to $\ord a_d + dn$ where $d > 4$. For all other $x$ in $K^*$, the order of that expression is at least
    \[
      M_2 := \min\,\{\,\ord a_i + i M_1 \mid 0 < i \leq d\,\}.
    \]
    Nora only needs to choose a non-zero element which has its order in
    \[
      \{\, n \mid n \not\equiv \ord a_d\!\pmod{d}\, \}\ \cap\ \{\, n < M_2\, \}.
    \]
    The possibility of $a_d$ being the last unset{t}led coef{f}icient is reduced to that of choosing the constant coef{f}icient by using the trick of reverse polynomial amply explained in Proposi{t}ion~\ref{P:biquad}.\\[-0.2cm]
    
    If Nora has to choose $a_i$ for some $i \notin \{ 0, 1, d - 1, d \}$ at the end, she hopes
    \begin{equation}\label{E:hope}
      a_i \neq -x^{-i} ( a_d x^d + \cdots + a_{i + 1}x^{i + 1} + a_{i - 1}x^{i - 1} + \cdots + a_0 )
    \end{equation}
    for all $x \in K^*$. Let
    \[
      M_3 := \min\,\{\,( \ord a_j - \ord a_d ) / ( d - j ) \mid 0 \leq j < d,\ j \neq i,\ a_j \neq 0 \,\}
    \]
    so that Nora can prevent all such $x$ with $\ord x < M_3$ from being roots by not allowing $\ord a_i$ to belong to $\{ \ord a_d + ( d - i )n \mid n \in \Z \}$. As $d - i$ is at least $2$, such choices are feasible. Next, we consider
    \[
      M_4 := \max\,\{\,\frac{\ord a_0 - \ord a_j}{j} \mid 0 < j < d,\ j \neq i,\ a_j \neq 0\,\}.
    \]
    For $x \in K^*$ with $\ord x > M_4$, the leading term of the right side of~\eqref{E:hope} is dictated by $a_0x^{-i}$ and thereby, has order of the form $\ord a_0 - in$ for some $n > M_4$. We recall that both $i$ and $d - i$ are at least $2$. Moreover, at least one of $d - i$ or $i$ has to be strictly greater than $2$ as $d > 4$. Of the lot that is yet to be accounted for, the order of the right side expression in~\eqref{E:hope} will be bounded from below by
    \[
      M_5 := \min\,\{\,\ord a_j + (j - i)n \mid M_3 \leq n \leq M_4,\ 0 \leq j \leq d,\ j \neq i,\ a_j \neq 0\,\}.
    \]
    This minimum exists and advises Nora to choose an $a_i$ with its order not in
    \[
      \{\,\ord a_d + ( d - i )n \mid n \in \Z\,\}\ \cup\ \{\,\ord a_0 - in \mid n \in \Z\,\}\ \cup\ \{\,n \geq M_5\,\}.
    \]
    If $d - i = i$, both of them have to be at least $3$ and the two arithmet{i}c progress{i}ons $\{\,\ord a_0 - in\,\} \cup \{\,\ord a_d + ( d - i )m\,\}$ will leave out enough many negat{i}ve integers as order opt{i}ons for Nora. If not, one of those two progressions has a greater common di{f}ference than the other and Nora is home.
  \end{proof}
  We are now able to provide a family of examples asked for by~\citeauthor*{GWZ18}. In~\citep{GWZ18}, they wonder about domains $D_1$ and $D_2$ such that the players choose coef{f}icients from $D_1$ while the roots are to be sought (avoided) in $D_2$.
  \begin{defn}[\protect{\cite[Problem~245]{Gou97}}]
    The \emph{maximal unrami{f}ied extension} of $\Q_p$, denoted by $\Q_p^{\text{unr}}$, is the union of all extensions of $\Q_p$ obtained by adjoining the $d$-th roots of unity whenever $d$ is coprime to $p$.
  \end{defn}
  It is an inf{i}nite extension of $\Q_p$ with $\overline{\F_p}^{\text{alg}}$ as its residue f{i}eld while the integer $p$ cont{i}nues to play the role of a uniformiser in $\Q_p^{\text{unr}}$. The map $\ord : \Q_p^{\text{unr}} \rightarrow \R$ still takes values in the subring of integers alone.
  \begin{corollary}\label{C:unr}
    Let $d \in \N^* \setminus \{ 1, 3, 4 \}$ and $\{ p_1, \ldots, p_k \}$ be any given f{i}nite set of rat{i}onal primes. If our players are required to choose the polynomial coef{f}icients from \Q, then Nora playing last can ensure that the rat{i}onal polynomial of degree $d$ thus constructed has no roots in any of the $\Q_{p_j}^{\text{unr}}$ for $j$ ranging from $1$ to $k$.
  \end{corollary}
  \begin{proof}
    For each $j$, Nora follows the winning strategy suggested to her by the the relevant Lemma~\ref{L:quad} or Theorem~\ref{Th:dgeq5} with $\K = \Q_{p_j}^{\text{unr}}$. At the end, she picks a negat{i}ve power $-k_j$ of $p_j$ allowed there and declares her choice to be $\sum_j p_j^{-k_j}$ which has all the necessary propert{i}es from her perspect{i}ve.
  \end{proof}
  The phenomenon cont{i}nues to work for $d = 4$ when no roots in exactly one $\Q_p^{\text{unr}}$ is being demanded. However, for two dist{i}nct primes $p_1$ and $p_2$, Wanda may conspire so that the $p_1$-adic orders of $a_0$ and $a_4$ in~\eqref{E:par} have the same parity modulo $2$ while their $p_2$-adic orders don't. This might confuse Nora's response.\\[-0.1cm]
  
  Note that for any f{i}nite set $\{ p_j \}$ consist{i}ng of primes and $d > 1$, Eisenstein irreducibility criterion (see~\cite[Proposi{t}ion~5.3.11]{Gou97}) can give us inf{i}nitely many monic polynomials with degree $d$ and integer coef{f}icients which are irreducible over all of those $\Q_{p_j}$'s. The polynomials constructed in our game during the course of Nora's victory need not always belong to the Eisenstein family but might not be irreducible either. Here, we would like to ment{i}on that a lot of ef{f}ort has gone into obtaining efficient algorithms for factorization of polynomials over locally compact f{i}elds. \citeauthor{Chi91} was the f{i}rst one to give a polynomial-t{i}me algorithm in this set{t}ing. We refer to \cite{Chi91,Coh93,CD00} and the references therein for def{i}nit{i}ons and learning more about this subject.\\[-0.1cm]
  
  Funct{i}on f{i}elds help to showcase another family of examples:
  \begin{corollary}\label{C:FF}
    Let $d \in \N^* \setminus \{ 1, 3 \}$ and $D_1 = \F((T)) $ while $D_2$ be a sub-ring of $\overline{F}^{\text{alg}}((T))$. If both Nora and Wanda choose polynomial coef{f}icients in $D_1$ with Nora playing last, she can ensure a victory with no roots in $D_2$.
  \end{corollary}
  \begin{proof}
    It is suf{f}icient to study the situat{i}on when $D_2$ is the whole of $\overline{F}^{\text{alg}}((T))$. Take $\K = D_2$. Nora follows the winning strategies prescribed in Lemma~\ref{L:quad}, Proposi{t}ion~\ref{P:biquad} or Theorem~\ref{Th:dgeq5} always taking care to choose her Laurent series in $\F((T))$ with correct orders in the transcendental variable $T$.
  \end{proof}
  
  \subsection{Cubic polynomials}\label{SSec:cub}
    We will f{i}nish this sect{i}on with a detailed discussion on polynomials with degree equal to $3$. For us, this was perhaps the hardest to understand. It is partly because the quest{i}on of having roots becomes one with the quest{i}on of reducibility for cubic polynomials.\\[-0.2cm]
    
    Let $f$ be a polynomial of degree $d$ with coef{f}icients in $K$ and a non-zero constant term. The \emph{Newton polygon $\mathcal{N}_f$} of $f$ is def{i}ned to be the lower boundary of the convex hull of the following collect{i}on of points:
    \begin{equation}
      \{ ( k, \ord a_k ) \mid a_k \neq 0 \} \}.
    \end{equation} It is, thereby, a cont{i}nuous and piecewise linear map from the closed interval $[0, d] \rightarrow \R$ with di{f}ferentiability breaking down only at some integer points.
    \begin{lemma}[\protect{cf.~\cite[Theorem~6.4.7]{Gou97}}]\label{L:slo}
      If $f(x) = 0$ for some $x \in \overline{\K}^{\text{alg}}$, then
      \[
        \ord x = - \mathcal{N}'_f (t)\ \text{for some}\ t \in [0, d] \setminus \Z.
      \]
    \end{lemma}
    \noindent It is also known that the Newton polygon of any irreducible polynomial is a line segment over the closed interval $[ 0, d ]$. We refer to~\citep[Chapter~6, \S\,3]{Cas86} for a proof.
    \begin{prop}\label{P:cub}
      Let $d = 3$ and $K$ be such that its residue f{i}eld $\O / \mathfrak{p}$ is f{i}nite. Then, the last player has a winning strategy.
    \end{prop}
    \begin{proof}
      The only case to be dealt with is when Nora is the last player. If  Wanda as Player\,\Rmnum{1} tries to choose either of $a_1$ or $a_2$ on her f{i}rst move, Nora can def{i}nitely manage to have one of $a_0$ or $a_3$ for herself to choose last. We have seen earlier in~\eqref{E:biquad} and \eqref{E:a0} how this favours Nora to be the winner of the game.\\[-0.2cm]
      
      Therefore, assume that Wanda chooses $a_3$ to be some non-zero element f{i}rst. Nora immediately lets $a_2 = 0$. Wanda will again prefer to pick $a_0$ on her second move. This choice should also be such that $a_0 / a_3$ is a perfect cube in $\K^*$ or else, Nora can take $a_1 = 0$ and win. In part{i}cular, we have that $\ord a_3 \equiv \ord a_0\pmod{3}$. It is plain that the two extreme vert{i}ces of $\mathcal{N}_f$ are $(0, \ord a_0 )$ and $( 3, \ord a_3 )$. Nora is forced to choose $a_1$ with order at least
      \[
        \ord a_0 + \frac{\ord a_3 - \ord a_0}{3}
      \]
      because~\citep[Chapter~6, \S\,3]{Cas86} ment{i}oned above. If this happens, the Newton polygon has a constant slope direct{i}ng that each of the roots of $f$ will have order exactly equal to $( \ord a_0 - \ord a_3 ) / 3$ by Lemma~\ref{L:slo}. Suppose that the associated monic polynomial factorizes as
      \[
        x^3 + \frac{a_1}{a_3}x + \frac{a_0}{a_3} = ( x + c ) ( x^2 + ax + b )
      \]
      for some $a \in K,\ b$ and $c$ in $K^*$. On comparing the coef{f}icients, one gets that $a + c = 0$ and $b = d/c$ where $d := a_0 / a_3$ which means
      \[
        \frac{a_1}{a_3} = \frac{d - c^3}{c}.
      \]
      We have $\ord d = 3\cdot\ord c = \ord a_0 - \ord a_3$ from Lemma~\ref{L:slo}. The conclusion is that the most signif{i}cant term of $a_1/a_3$, on expanding as a Laurent series in $\rho$, should be of the form $\rho^{\ord c - \ord d}( d - c^3 ) /c\pmod{\mathfrak{p}}$ if $f$ were to factorize over \K. Recall that $\rho$ is a uniformiser in \O.\\[-0.2cm]
      
      As $d = a_0/a_3$ has been assumed to be a cube in $K$, we also know that $d_0 := \rho^{-\ord d} d\pmod{\mathfrak{p}}$ is a non-zero cube in the f{i}nite f{i}eld $\O / \mathfrak{p}$. Once such a cube $d_0 \in ( \O / \mathfrak{p} )^*$ has been f{i}xed, we have at most $\abs{\O/\mathfrak{p}} - 1$ elements in the set
      \[
        \{\,( d_0 - c_0^3 ) / c_0 \mid c_0 \in ( \O / \mathfrak{p} )^*\,\}
      \]
      at least one of which equals zero. Thus, it will miss some non-zero element $\widetilde{a}_1 \in \O/\mathfrak{p}$ (say). Nora chooses $a_1$ to be such that $a_1 / a_3$ has order equal to $\ord d - \ord c = (2/3)\cdot\ord d$ and
      \[
        \rho^{\ord c - \ord d} ( a_1 / a_3 ) \equiv \widetilde{a}_1\pmod{\mathfrak{p}}.
      \]
      This ensures that the polynomial $f$ cannot factorize over \K.\\[-0.2cm]
      
      If Wanda had begun with $a_0$ instead, Nora will try to win for the reverse polynomial so that the roles of $a_0$ and $a_3$ ($a_1$ and $a_2$) are interchanged.
    \end{proof}
    Not{i}ce that the above game could have been played and won by either of the last players with rat{i}onal coef{f}cients of the polynomial $f$ when looking for $p$-adic numbers as roots.
    \begin{corollary}\label{C:cub}
      Let $p$ be a prime, $d = 3$ and the players be required to choose coef{f}cients from \Q. Then, Nora playing last can ensure that the rat{i}onal polynomial obtained in the end does not have a root in any ring $D_2 \subset \Q_p$.
    \end{corollary}
    \begin{proof}
      As in the proof of Corollary~\ref{C:FF}, we may regard $D_2 = \Q_p$ without any loss of generality. Nora follows the strategy outlined in Proposi{t}ion~\ref{P:cub} above while simultaneously ensuring that the elements chosen are all rat{i}onal numbers of appropriate $p$-adic orders. This is easy enough for her.
    \end{proof}
    It is plausible that our strategy for Proposi{t}ion~\ref{P:cub} may also work when the residue f{i}eld $\O / \mathfrak{p}$ is isomorphic to $\mathbb{Q}$. We must point out that in order to win the game over \K, Nora is required to f{i}nd an irreducible cubic polynomial having shape $y^3 + by - d_0$ with coef{f}icients in $\O/\mathfrak{p}$ and a non-zero cube $d_0$ given to her.\\[-0.1cm]
    
    The game t{i}lts in Wanda's favour if the polynomial coef{f}icients are to be chosen from the ring of integers. The reader is reminded that for a polynomial to have a root in \K, or equivalently, a linear expression as its factor, it is suf{f}icient that its Newton polygon have some slope of length one. This is because all the roots of any irreducible factor of $f$ are of the same order.
    \begin{prop}
      Let $d = 3$ and the polynomial coe{f}ficients be chosen from some complete discrete valuat{i}on ring \O. Then, Wanda can always ensure that it has roots in the f{i}eld of fract{i}ons \K.
    \end{prop}
    \begin{proof}
      She can always win by Lemma~\ref{L:GWZ} if she is the last player. Else, Wanda is Player\,\Rmnum{1}. She begins with declaring $a_0 = \rho$. If Nora does not choose $a_1$ to have a posi{t}ive order next, Wanda can ask for $a_1$ to be $1$ on her second move and ensure a slope of length one in the Newton polygon. This is also true if Nora picks $a_1$ to be a unit in \O.\\[-0.2cm]
      
      If Nora makes sure to have $a_1$ with $\ord a_1 > 0$, Wanda lets $a_2 = 1$. Irrespect{i}ve of Nora's subsequent closing move, there will be a slope of length one in the Newton polygon associated with the polynomial.
    \end{proof}
   An object{i}on may be raised with regards to Nora's capability to check for $a_0/a_3$ being a cube in $K$. When $\operatorname{char}\,( \O/\mathfrak{p} )$ is not $3$, Nora can use the statement below (see for example~\citep[Chapter~10]{AM69}) to reduce this quest{i}on to checking cubicity in the residue f{i}eld.
    \begin{lemma}[Hensel's lemma]
      If $f \in \O[X]$ and $\alpha_0 \in \O$ is such that $f(\alpha_0) \equiv 0$ modulo $f'(\alpha_0)^2\mathfrak{p}$, then there exists an $\alpha \in K$ with $f(\alpha) = 0$ and $\alpha \equiv \alpha_0$ modulo $f'(\alpha_0)\mathfrak{p}$. Such an $\alpha$ is also unique provided $\alpha_0 \neq 0$.
    \end{lemma}
    Either player can apply this to the polynomial $X^3 - ( \rho^{-\ord ( a_0 / a_3 )} a_0 / a_3 )$. As long as $\operatorname{char} ( \O/\mathfrak{p} ) \neq 3$, this \O-polynomial has a root in \O\ if{f} $3$ divides $\ord ( a_0 / a_3 )$ and $\rho^{-\ord ( a_0 / a_3 )} a_0 / a_3$ is a cube in $\O/\mathfrak{p}$. The issue of availability of $q$-th roots in $p$-adic f{i}elds has been considered more elaborately than here in~\cite{MS13}. Before ending this sect{i}on, a di{f}ferent proof is presented in characterist{i}c $3$ where the analysis is simpler.
    \begin{lemma}\label{L:char3}
      Let $d = 3$ and \K\ have characterist{i}c $3$. Then, the last player is able to win.
    \end{lemma}
    \begin{proof}
      As before, we bother about Nora alone. Wanda must pick a non-zero $a_3$ (or $a_0$) on her f{i}rst move. This is followed by Nora taking $a_2$ to be zero. Wanda's second choice should be that of $a_0$ and such that $a_0 / a_3 = d^3$ for some $d \in \K^*$. The monic polynomial
      \[
        x^3 + ( a_1 / a_3 )x + ( a_0 / a_3 ) = x^3 + ( a_1 / a_3 ) x + d^3
      \]
      transforms as
      \[
        ( x + d )^3 + ( a_1 / a_3 )x
      \]
      because $\operatorname{char} \K = 3$. Hence, Nora would want
      \[
        a_1 \neq -a_3 x^{-1} ( x + d )^3
      \]
      for all $x \in \K^*$. For $x$ such that $\ord x < \ord d$, the order of the right side expression is $\ord a_3 + 2\cdot\ord x$. For elements of $\K^*$ with $\ord x > \ord d$, we have its order to be $\ord a_0 - \ord x$. If $\ord x$ equals $\ord d$, the order of $a_3x^{-1}(x + d )^3$ should be at least $\ord a_3 + 2\cdot\ord d$ and be equivalent to $\ord a_3 - \ord d$ modulo $3$. Nora can choose $a_1$ to have order larger than the maximum of $\ord a_3 + 2\cdot\ord d$ and $\ord a_0 - \ord d$ with $\ord a_1 \not\equiv \ord a_3 - \ord d\pmod{3}$. This will translate into a victory for her.
    \end{proof}
    The residue f{i}eld $\O/\mathfrak{p}$ is allowed to have inf{i}nite cardinality over here. For $\K = \F_q ((T))$ where $q$ is some power of $3$, an element $\alpha$ is a cube in \K\ if{f} the non-zero coef{f}icients of $\alpha$ are those where the corresponding power of $T$ is an integer divisible by $3$. Otherwise said, $\alpha$ needs to belong to $\F_q (( T^3 ))$.

\section{\protect{Maximal abelian extension of \Q}}\label{S:cyc}
    A speci{f}ic quest{i}on was asked by~\citeauthor*{GWZ18} with regards to rat{i}onal polynomials. Because Lemma~\ref{L:GWZ}, it can be rephrased as to whether Nora can choose her coef{f}icients so that the result{i}ng polynomial $f \in \Q [x]$ has no roots in the compositum of all solvable extensions of \Q. As is well-known, this is impossible for $d < 5$. In part{i}cular, the Galois group is always abelian for quadrat{i}c polynomials being a subgroup of $\Z / 2\Z$. We next see that Nora can ensure the counterpart in the cubic case to be not so.
    \begin{lemma}\label{L:cnab}
      Let $d = 3$ and the players required to choose coef{f}icients to be rat{i}onal numbers. Then, Nora playing last can ensure that the polynomial obtained at the end does not have a root in any abelian extension of \Q.
    \end{lemma}
    \begin{proof}
      It was already seen in~\S\S\,\ref{SSec:cub} that Nora can choose the last coef{f}icient to be such that the rat{i}onal polynomial is irreducible over \Q, with the help of some $p$-adic f{i}eld.\\[-0.2cm]
      
     If this last coef{f}icient is $a_0$, Nora uses Corollary~\ref{C:cub} to obtain the correct $p$-adic order and leading term in the $p$-adic expansion for the coef{f}icient $a_0/a_3$ in the monicized rat{i}onal polynomial
      \[
        x^3 + (a_2/a_3)x^2 + (a_1/a_3)x + (a_0/a_3).
      \]
      This gives her that the polynomial obtained will have no roots in $\Q_p$ and a fort{i}ori no roots in \Q\ implying that it is irreducible over rat{i}onal numbers. The discriminant of the monic polynomial $x^3 + Ax^2 + Bx + C$ is given by
      \[
        A^2B^2 - 4B^3 - 4A^3C -27C^2 + 18ABC.
      \]
      Once $A$ and $B$ have been f{i}xed, the expression can be made negat{i}ve by choosing a suf{f}iciently large rat{i}onal number $C$. Since the $p$-adic order of $a_0/a_3$ is dictated to use by Corollary~\ref{C:cub}, this can be achieved by appending a large power of $p$. This, in part{i}cular, means that the discriminant of our cubic polynomial will not be a square in \Q. The associated Galois group will then be the whole of the symmetric group $S_3$ by~\cite[\S\,14.6]{DF04}. Nora does the same for the reverse polynomial of $f$ when $a_3$ is the last coef{f}icient to be chosen by her.\\[-0.2cm]
      
     Therefore, Wanda will like to prevent the leading or the constant coef{f}icient to be chosen by Nora in the end. Wanda must begin with picking $a_3$ (or $a_0$) in order to do so. This is followed by Nora declaring $a_2 = 0$ and Wanda making a non-zero choice of $a_0$ (or $a_3$) thereaf{t}er. Consider the monicized rat{i}onal polynomial
      \[
        x^3 + (a_1/a_3)x + (a_0/a_3)
      \]
      as belonging to $\Q_{p}[x]$, where $a_1$ is yet to be determined by Nora. She learns the correct $p$-adic order of $a_1/a_3$ from Corollary~\ref{C:cub} so as to have the polynomial to be irreducible over $\Q_p$ (and \Q\ as well). The discriminant of this depressed cubic equat{i}on is given as
      \[
        - 4 (a_1/a_3)^3 - 27 (a_0/a_3)^2.
      \]
      If Nora had further chosen $a_1$ such that $a_1 / a_3$ is a posi{t}ive rat{i}onal number, this discriminant would be negat{i}ve and our polynomial cannot have roots in any abelian extension of \Q\ once again.
    \end{proof}
    For games where the degree of the polynomial has been f{i}xed to be greater than $8$, Nora f{i}rst f{i}nds an odd prime $p$ such that $( p - 1 )/2$ has all its prime factors to be larger than $d$. The Chinese remainder theorem helps her to have a number $N$ such that $N \equiv 3\pmod{4}$ and $N \equiv 2\pmod{p_i}$ for all odd primes $p_i \leq d$. The rest is done by applying Dirichlet's theorem on primes in arithmet{i}c progressions to the sequence $\{\,N + t \cdot 4p_2 \cdot \ldots \cdot p_{\pi (d)} \mid t \in \N\,\}$, where $p_i$ denotes the $i$-th prime and $\pi (d)$ denotes the number of primes less than or equal to $d$.\\[-0.1cm]
    
    Nora will then employ an enhanced strategy to avoid roots in $\Q_{p}^{\text{unr}} ( \sqrt{-p} ) = \Q_{p}^{\text{unr}} ( \sqrt{p} )$. The reason is that $\Q_{p}^{\text{unr}}$ contains the subf{i}eld
    \begin{equation}
      \overline{\Q}^{\text{ab} \setminus p} := \bigcup_{\protect{\substack{d\,>\,1,\\ p\,\nmid\,d}}} \Q ( \zeta_d )
    \end{equation}
    where $\zeta_d$ denotes a primi{t}ive $d$-th root of unity, while the maximal abelian extension of \Q\ is given by
    \begin{equation}\label{E:11}
      \overline{\Q}^{\text{ab}} = \bigcup_{r \in \N} \overline{\Q}^{\text{ab} \setminus p} ( \zeta_{p^r} ).
    \end{equation}
    The degree of each of the f{i}nite extensions in~\eqref{E:11} is $(p-1)p^{r - 1}$ over $\overline{\Q}^{\text{ab} \setminus p}$. Thus, any element in the maximal abelian extension is either contained in the quadrat{i}c extension $\overline{\Q}^{\text{ab} \setminus p} (\sqrt{p})$ or has degree $> d$ over $\overline{\Q}^{\text{ab} \setminus p}$ and, thereby, at least that much over the f{i}eld of rat{i}onal numbers too. It is then clear that any root of a degree $d$ rat{i}onal polynomial which belongs to $\overline{\Q}^{\text{ab}}$ can only be in $\overline{\Q}^{\text{ab} \setminus p} (\sqrt{p})$, if at all. The lat{t}er is contained in $\Q_{p}^{\text{unr}} (\sqrt{p)}$ and avoiding roots in this ultrametric f{i}eld will be enough for our polynomial to not have roots in any abelian extension of \Q. The reader may check that the $p$-adic valuat{i}on map takes values in the addi{t}ive group $(1/2)\Z$ on this quadrat{i}c extension.\\[-0.1cm]
    
    Nora will also require the following result about arithmet{i}c progressions in half-integers:
    \begin{lemma}\label{L:AP}
      Let $n_1, n_2 \in \Z,\ d > 8$ and $2 < i < d - 2$. Then, there exists a sequence of integers going to $-\infty$ which is disjoint from both sets $\{ n_1 + ni / 2 \mid n \in \Z \}$ and $\{ n_2 + n(d - i) / 2 \mid n \in \Z \}$.
    \end{lemma}
    \begin{proof}
      It suf{f}ices to show that for any $N \in \Z$, we have an integer $N_0 < N$ which is in neither of those two sets. The roles of $i$ and $d - i$ are interchangeable within this proof. Moreover, we have arranged that $\min \{ i/2, (d-i)/2 \} \geq 3/2$.\\[-0.2cm]
      
      When it achieves this minimum for some $i$ (say), then the other term should be at least $3$ as $d > 8$. Pick $m_1 \in \{ n_1 + 3n/2,\ n \in \Z \}$ to be less than $N$ and a non-integer. If $m_1$ also belongs to $\{ n_2 + n(d-i)/2,\ n \in \Z \}$, then $N_0 := m_1 - (1/2)$ does not belong to either of the two sets and we are done. If instead $m_1 - (1/2) \in \{ n_2 + n(d-i)/2,\ n \in \Z \}$, then $N_0$ is taken to be $m_1 - (5/2)$. Else, we may have $N_0$ to be $m_1 - (1/2)$ itself.\\[-0.2cm]
      
      When $\min \{ i/2, (d-i)/2 \} = 2$, we again have that $\max \{ i/2, (d-i)/2 \} \geq 5/2$. Pick some $m_1 < N$ from $\{ n_1 + 2n,\ n \in \Z \}$. If it is also in $\{ n_2 + n(d-i)/2,\ n \in \Z \}$, then take $N_0 = m_1 - 1$. Otherwise one of the integers $m_1 - 1$ or $m_1 - 3$ will do the job for us.\\[-0.2cm]
      
      For $d$ and $i$ such that $\min \{ i/2, (d-i)/2 \} \geq 5/2$, the arguments are only easier than the ones outlined so far.
    \end{proof}
    We are now ready to assert that
    \begin{prop}
      Let the degree of the polynomial be greater than $8$ and the players be required to choose rat{i}onal coef{f}icients. Then, Nora can ensure that the polynomial has no roots in $\overline{\Q}^{\text{ab}}$ when she is the last player.
    \end{prop}
    \begin{proof}
      As $d$ is greater than $8$, Nora gets at least four moves before her last one. She uses those to have $a_1,\ a_2,\ a_{d - 2}$ and $a_{d - 1}$ f{i}xed by either of the players before Nora's last move of the game. If this results in her choosing the constant coef{f}icient at the end, she will want an $a_0 \in \Q^*$ such that
      \begin{equation}\label{E:Qpunr}
      a_0 \neq - ( a_dx^d + \ldots + a_1x )\quad\text{for all}\quad x \in \Q_{p}^{\text{unr}} (\sqrt{p)} \setminus \{ 0 \}
      \end{equation}
      when $a_3, \ldots, a_1 \in \Q$ have been decided before and $a_3 \neq 0$. The odd prime $p$ is such that all prime factors of $( p - 1 ) /2$ are greater than $d$. For non-zero elements $x$ in $\Q_p^{\text{unr}} (\sqrt{p})$ such that $\ord_p{x} < M_1$ where
      \begin{equation}
        M_1 := \min \big\{ ( \ord_p a_i - \ord_p a_d ) / ( d - i ) \mid 0 < i < d,\ a_i \neq 0 \big\},
      \end{equation}
      the right-side expression in~\eqref{E:Qpunr} has its order lying in the set $\{ \ord a_3 + 3n \mid n \in (1/2)\Z,\ n < M_1 \}$. The order of that expression is bounded below by
      \[
        M_2 := \min \{ \ord_p a_i + iM_1 \mid i = 1, 2, 3 \}
      \]
      for all remaining $x$. Nora chooses a rat{i}onal number whose $p$-adic order belongs to $\Z \setminus \big(\,\{ \ord_p a_d + dn \mid n \in (1/2)\Z,\ n < M_1 \} \cup \{ n \geq M_2 \}\,\big)$. She does the same for the reverse polynomial if $a_d$ is to be picked last by her. By our observat{i}ons above, such a polynomial will not have roots in any abelian extension of \Q.\\[-0.2cm]
      
      Next, we consider the situat{i}on when Nora has to choose $a_i$ for some $2 < i < d - 2$ in the end. She desires
      \[
        a_i \neq -x^{-i} ( a_d x^d + \cdots + a_{i + 1}x^{i + 1} + a_{i - 1}x^{i - 1} + \cdots + a_0 ) \text{ for all } x \in \Q_p^{\text{unr}} (\sqrt{p})^*.
      \]
      Let $M_3 = \min\,\{\,( \ord a_j - \ord a_d ) / ( d - j ) \mid 0 \leq j < d,\ j \neq i,\ a_j \neq 0 \,\}$ and $M_4 = \max\,\{\,( \ord a_0 - \ord a_j ) / j  \mid 0 < j < d,\ j \neq i,\ a_j \neq 0\,\}$ so that Nora can avoid all $x$ with $\ord_p x \notin [ M_3, M_4 ]$ from being roots by not allowing $a_i$ to belong to either of the bi-inf{i}nite arithmet{i}c progressions
      \begin{equation}\label{E:AP}
        \{ \ord_p a_d + n( d - i ) / 2 \mid n \in \Z \} \text{ and } \{ \ord_p a_0 - ni/2 \mid n \in \Z \}.
      \end{equation}
      The elements $x \in \Q_p^{\text{unr}} (\sqrt{p})^* \setminus \{ 0 \}$ with $M_3 \leq \ord_p x \leq M_4$ will lead to expressions on the right side of~\eqref{E:hope} with a uniform lower bound on their $p$-adic order, as before. Lemma~\ref{L:AP} tells us that Nora can f{i}nd an integer $N_0$ less than this lower bound and not in either of the arithmet{i}c progressions in~\eqref{E:AP}. On choosing $a_i$ with $\ord_p a_i = N_0$, the polynomial obtained will have no roots in any abelian extension of \Q.
    \end{proof}
    
  \section*{Acknowledgments}  
    \noindent We thank Prof.~N.~Saradha for her careful suggest{i}ons which have helped in improving the presentat{i}on of this paper. Samuel Zbarsky went through a preliminary version of the manuscript thoroughly and kindly pointed out various mistakes. The f{i}rst author acknowledges the support of the OPERA award and the Research Ini{t}iat{i}on Grant of BITS Pilani. The second author thanks Drs.~Jinxin Xue, Min Wang and Qijun Yan for providing infrastructural support during the period of this work.
    
\bibliographystyle{plainnat}
\bibliography{references_arXiv}

\end{document}